\documentclass[11pt]{article}
\usepackage{amsmath}
\usepackage{amsthm}
\usepackage{amsfonts}
\usepackage{amssymb}
\usepackage{latexsym}
\usepackage{diagbox}
\usepackage{tikz}
\usepackage{tkz-graph}

\usepackage[colorlinks,
linkcolor=blue,
anchorcolor=blue,
citecolor=blue
]{hyperref}

\newtheorem{thm}{Theorem}

\DeclareMathOperator{\asc}{asc}
\DeclareMathOperator{\des}{des}

\title{Non-overlapping descents and ascents in stack-sortable permutations}

\author{Sergey Kitaev\footnote{Department of Mathematics and Statistics, University of Strathclyde, 26 Richmond Street, Glasgow G1 1XH, United Kingdom. 
{\bf Email:} sergey.kitaev@strath.ac.uk.}\ \ and Philip B. Zhang\footnote{College of Mathematical Science, Tianjin Normal University, Tianjin  300387, P. R. China  {\bf Email:} zhang@tjnu.edu.cn.}}

\begin{document}

\maketitle 

\noindent\textbf{Abstract.} 
The Eulerian polynomials $A_n(x)$ give the distribution of descents over permutations. It is also known that the distribution of descents over stack-sortable permutations (i.e.\ permutations sortable by a certain algorithm whose internal storage is limited to a single stack data structure) is given by the Narayana numbers  
$\frac{1}{n}{n \choose k}{n \choose k+1}$. On the other hand, as a corollary of a much more general result, the distribution of the statistic ``maximum number of non-overlapping descents'', MND, over all permutations is given by $\sum_{n,k\geq 0}D_{n,k}x^k\frac{t^n}{n!}=\frac{e^t}{1-x(1+(t-1)e^t)}$. 

In this paper, we show that the distribution of MND over stack-sortable permutations is given by $\frac{1}{n+1}{n+1\choose 2k+1}{n+k \choose k}$. We give two proofs of the result via bijections with rooted plane (binary) trees allowing us to control MND. Moreover, we show combinatorially that MND is equidistributed with the statistic MNA, the maximum number of non-overlapping ascents, over stack-sortable permutations. The last fact is obtained by establishing an involution on stack-sortable permutations that gives equidistribution of 8 statistics. \\

\noindent {\bf AMS Classification 2010:} 05A15


\section{Introduction}\label{intro-sec}

A permutation of length $n$ is a rearrangement of the set $[n]:=\{1,2,\ldots,n\}$. Denote by  $S_n$  the set of permutations of $[n]$ and let $\varepsilon$ be the empty permutation. For a permutation $\pi=\pi_1\pi_2\cdots \pi_n$ with $\pi_i=n$,  the {\em stack-sorting operator} $\mathcal{S}$  is defined recursively as follows, where $\mathcal{S}(\varepsilon)=\varepsilon$,
$$\mathcal{S}(\pi) = \mathcal{S}(\pi_1\cdots \pi_{i-1})\,  \mathcal{S}(\pi_{i+1}\cdots \pi_n)\, n.$$
A permutation $\pi$ is {\em stack-sortable} (S-S) if $\mathcal{S}(\pi) =12\cdots n$. Let $SS_n$  denote the set of S-S permutations of length $n$. The stack-sorting operator appears in numerous studies in the mathematics and theoretical computer science literature (e.g.\ see \cite{Bona} and references therein). 

The permutations in $SS_n$ are counted by the Catalan numbers $C_n=\frac{1}{n+1}\binom{2n}{n}$, and they are precisely  the set of 231-avoiding permutations, where a permutation $\pi_1\pi_2\cdots\pi_n$ avoids a pattern $p=p_1p_2\cdots p_k$ (which is also a permutation) if there is no subsequence $\pi_{i_1}\pi_{i_2}\cdots\pi_{i_k}$ such that $\pi_{i_j}<\pi_{i_m}$ if and only if $p_j<p_m$ \cite{Kitaev2011Patterns,Knuth1969Art}.  Let $S_n(p)$ denote the set of $p$-avoiding permutations of length $n$. Then $SS_n=S_n(231)$. S-S permutations have a nice recursive structure: if $\pi=AnB\in S_n(231)$  then $A<B$ (i.e.\ every element in $A$ is less than any element in $B$) and $A$ and $B$ are (possibly empty) 231-avoiding permutations independent from each other. 

For a permutation $\pi=\pi_1 \pi_2 \cdots \pi_n$,  the {\em descent} (resp., {\em ascent}) {\em statistic} on $\pi$, $\des(\pi)$ (resp., $\asc(\pi)$), is defined as the number of $i\in [n-1]$ such that $\pi_i > \pi_{i+1}$ (resp., $\pi_i<\pi_{i+1}$). For example, des(562413)$=$2 and asc(35124)$=$3. The distribution of descents (or ascents) over $S_n$ is given by the  {\em Eulerian polynomial} 
\begin{align*}
A_n(x) := \sum_{\pi \in S_n}x^{\des (\pi)}=\sum_{k=1}^{n}k!S(n,k)(x-1)^{n-k}
\end{align*}
where $S(n,k)$ is the {\em Stirling numbers of the second kind}. On the other hand, the distribution of descents over $SS_n$ (i.e.\ 231-avoiding permutations) is given by the  Narayana numbers  $$\frac{1}{n}\binom{n}{k}\binom{n}{k+1}$$ (see  \cite{Stanley1989}). To complete the picture, the distribution of descents over 123-avoiding permutations is given by the following formula \cite{BBS2010,BKLPRW}, where $t$ and $x$ correspond to the length and the number of descents:
\begin{align}
\frac{-1+2tx+2t^2x-2tx^2-4t^2x^2+2t^2x^3+\sqrt{1-4tx-4t^2x+4t^2x^2}}{2tx^2(tx-1-t)}. \notag
\end{align}
The distribution for the number of  321-avoiding permutations of length $n$ with $k$ descents \cite[A091156]{oeis} is
\begin{align}\label{321-des}
\frac{1}{n+1}{n+1\choose k}\sum_{j=0}^{n-2k}{k+j-1\choose k-1}{n+1-k \choose n-2k-j}. 
\end{align}

The notion of the {\em maximum number of non-overlapping occurrences} of  a {\em consecutive pattern} (that is, occurrences of a pattern defined above with the requirement of $i_2=i_1+1$, $i_3=i_2+1$, etc) in a permutation has been considered in \cite{Kitaev2005}. It turns out that the distribution of the maximum number of non-overlapping occurrences of a consecutive pattern can be expressed in terms of the {\em exponential generating function} (e.g.f.) for the number of permutations {\em avoiding} the pattern. In particular, since the e.g.f. for permutations with no descents is clearly $e^t$, we have~\cite{Kitaev2005} that 
$$\sum_{n,k\geq 0}D_{n,k}x^k\frac{t^n}{n!}=\frac{e^t}{1-x(1+(t-1)e^t)}$$
where $D_{n,k}$ is the number of permutations with $k$ non-overlapping descents (but without $k+1$ non-overlapping descents). 

The main focus of this paper is the study of the {\em maximum number of non-overlapping descents} ({\em MND}) and {\em non-overlapping ascents} ({\em MNA}) over S-S permutations. For example, MND(13254)$=$MND(32154)=2 while 2=des(13254)$\neq$des(32154)=3. In Section~\ref{sec-MND-MNA} we will provide an involution on S-S permutations that preserves 8 statistics and shows that the distribution of MNA is the same as that of MND. Further, we will prove in two ways (in Sections~\ref{sec-binary-trees} and~\ref{sec-trees}), via establishing appropriate bijections with rooted plane (binary) trees, that the distribution of MND over S-S (i.e.\ 231-avoidable permutations) is given by 
\begin{align}\label{231-MND}
\frac{1}{n+1}{n+1\choose 2k+1}{n+k \choose k}.
\end{align}
The respective numbers are recorded in \cite[A108759]{oeis} and the formula is derived in \cite{Callan2005}.  Finally, in Section~\ref{open-directions} we suggest directions for further research. 

\section{Preliminaries}

In this section we provide necessary basic definitions and notation used in the paper.

\subsection{Trees.}\label{trees-sub}  A {\em rooted plane tree} or an {\em ordered tree}, consists of a set of vertices each of which has a (possibly empty) linearly ordered list of vertices associated with it called its {\em children}. One of the vertices of the tree is called the {\em root}. A vertex with no children is a {\em leaf}. A vertex in a rooted plane tree is {\em internal} if it is neither a leaf nor the root. Hence, the vertices are partitioned into three classes: root, internal vertices, leaves. Rooted plane trees can be produced as follows:
\begin{itemize}
\item A single vertex with no children is a rooted plane tree. That vertex is the root. 
\item If $T_1,\ldots,T_k$ is an ordered list of rooted plane trees with roots $r_1,\ldots,r_k$ and no vertices in common, then a rooted plane tree $T$ can be constructed by choosing an unused vertex $r$ to be the root, letting its $i$-th child be $r_i$ ($r_1,\ldots,r_k$ are not roots any more).
\end{itemize}

An internal vertex in a tree is {\em marked} if it has a leaf as its child.  The formula (\ref{231-MND}) gives the number of rooted plane trees with $n$ edges and $k$ marked vertices  \cite{Callan2005}. In this paper, we call these objects {\em marked trees}.  The same formula  counts  {\em full binary trees} (ordered trees where there are no or two children for each vertex) on $2n$ edges by the value $k$ of the following statistic $X$ described in \cite[A108759]{oeis}. Delete all right edges in a given full binary tree $T$ with $2n$ edges leaving the left edges in place. This partitions the left edges into line segments (or paths) of lengths say $\ell_1$, $\ell_2,\ldots,\ell_t$, with $\sum_{i=1}^{t} \ell_i = n$. Then $X(T) = \sum_{i=1}^{t} \lfloor \frac{\ell_i}{2} \rfloor$. This result is implicit in \cite{Sun2010}. For example, $X(T)=2$  for $T$ in Figure~\ref{example-thm-binary-fig}. Note that an interpretation of  (\ref{231-MND})  on {\em Dyck paths} can be given via a standard bijection with full binary trees \cite[A108759]{oeis}, but Dyck paths are not to be considered in this paper.

\subsection{Permutations.} Section~\ref{intro-sec} introduces the notion of a pattern-avoiding permutation, the structure of S-S permutations (i.e.\ 231-avoiding permutations) and the statistics asc, des, MNA and MND. In this paper, we  also need the notions of statistics from \cite{Kitaev2011Patterns} introduced in Table~\ref{tab-statistics}. For example, 
lmin(52341)$=$rmax(52341)$=$ rmin(413625)$=$3, ldr(526341)$=$rar(413625)$=$2. Moreover, the {\em reverse} of $\pi=\pi_1\pi_2\cdots\pi_n$ is the permutation $r(\pi)=\pi_n\pi_{n-1}\cdots\pi_1$ and the {\em complement} of $\pi$ is the permutation $c(\pi)$ obtained from $\pi$ by replacing each $\pi_i$ by $n+1-\pi_i$. For example, $r(31425)=52413$ and $c(31425)=35241$. For a sequence of distinct numbers $x=x_1x_2\cdots x_k$, the {\em reduced form} of $x$, red($x$), is obtained from $x$ by replacing the $i$-th smallest number by $i$. For example, red(2547)$=$1324. 

\begin{table}
\begin{center}
\begin{tabular}{l|l}
\hline
\hline
Stat & Definition on a permutation $\pi_1\pi_2\cdots\pi_n$\\
\hline
\hline
lmin & number of left-to-right minima = $|\{i\ |\ \pi_j>\pi_i \mbox{ for all }j<i\}|$\\
\hline
rmin & number of right-to-left minima = $|\{i\ |\ \pi_i<\pi_j \mbox{ for all }j>i\}|$ \\
\hline
ldr & length of leftmost decreasing run = max$\{i\ |\ \pi_1>\pi_2>\cdots>\pi_i \}$ \\
\hline
rar & length of rightmost ascending run = max$\{i\geq 1\ |\ \pi_n>\cdots >\pi_{n-i+1} \}$ \\
\hline
rmax & number of right-to-left maxima = $|\{i\ |\ \pi_i>\pi_j \mbox{ for all }j>i\}|$ \\
\hline

\end{tabular}
  \caption{Permutation statistics in this paper.}
\label{tab-statistics}
\end{center}
\end{table}

$k$-tuples of (permutation) statistics $(s_1,s_2,\ldots,s_k)$ and $(s'_1,s'_2,\ldots,s'_k)$ are equidistributed over a set $S$ if $$\sum_{a\in S}t_1^{s_1(a)}t_2^{s_2(a)}\cdots t_k^{s_k(a)}=\sum_{a\in S}t_1^{s'_1(a)}t_2^{s'_2(a)}\cdots t_k^{s'_k(a)}.$$  For example, (asc, des) and (des, asc) are equidistributed over $S_n$ for any $n\geq 0$, and this fact is trivial from applying the reverse (or complement) to all permutations in $S_n$. The fact that (MNA, MND) is equidistributed with (MND, MNA) over S-S permutations is not trivial, and it will follow from the more general equidistribution result (involving 8 statistics) in Section~\ref{sec-MND-MNA}.

\section{Equidistribution of MND and MNA over S-S permutations}\label{sec-MND-MNA}

The equidistribution of MND and MNA over S-S permutations follows from the following more general result. 

\begin{thm}\label{8-stat-equidistribution-thm} The following $8$-tuples of statistics are equidistributed over $S_n(231)$ for $n\geq 0$: 
$$\mbox{(MND, MNA, asc, des, ldr, rar, lmin, rmin)}$$
$$\mbox{(MNA, MND, des, asc, rar, ldr, rmin, lmin).}$$
\end{thm}

\begin{proof} Recall that if $\pi=AnB\in S_n(231)$  then $A<B$. We define the following recursive map $f$ on $S_n(231)$ that is easy to check by induction on $n$ to be an involution (that is, $f^2$ is the identity map):   
\begin{itemize}
\item $f(\varepsilon)=\varepsilon$.
\item For $\pi=AnB\in S_n(231)$, $n\geq 1$, where $A$ and $B$ are possibly empty, $f(\pi)=f(\mbox{red}(B))n\big(f(A)\big)^+$ where $\big(f(A)\big)^+$ is formed by the largest elements in $f(\pi)$ excluding $n$ and red($\big(f(A)\big)^+)=f(A)$. That is, $f$ recursively swaps the largest and smallest elements in $\pi$ below the largest element. It is easy to see by induction on $n$ that $f(\pi)\in S_n(231)$.   
\end{itemize}
We prove by induction on $n$ that the map $f$ respects the statistics. The base cases of $n\in\{0,1\}$ are trivial, so assume $n\geq 2$. We now apply the inductive hypothesis in the following three cases (some of which, but not all, result in the same derivations):

\noindent \fbox{$A\neq \varepsilon$, $ B\neq \varepsilon$.}

ldr($\pi$)=ldr($A$)=rar($f(A$))=rar($f(\pi)$). 

rar($\pi$)=rar(red$(B)$)=ldr($f($red$(B)$))=ldr($f(\pi)$). 

lmin($\pi$)=lmin($A$)=rmin($f(A$))=rmin($f(\pi)$). 

rmin($\pi$)=rmin(red$(B)$)=lmin($f($red$(B)$))=lmin($f(\pi)$). 

asc($\pi$)=asc($A$)+1+asc(red($B$))=des($f(A)$)+1+des($f($red($B$)))=des($f(\pi)$).

des($\pi$)=des($A$)+1+des(red($B$))=asc($f(A)$)+1+asc($f($red($B$)))=asc($f(\pi)$).

Now, suppose that ldr(red$(B)$) is even. Then, because ldr($\pi$)=rar($f(\pi)$), the element $n$ does not contribute to MND($\pi$) and MNA($f(\pi)$), so we have

MND($\pi$)=MND($A$)+MND(red($B$))=MNA($f(A)$)+MNA($f($red($B$)))

=MNA($f(\pi)$).

On the other hand, if ldr(red$(B)$) is odd, then $n$ contributes one extra non-overlapping descent in $\pi$ and one extra non-overlapping ascent in $f(\pi)$:

MND($\pi$)=MND($A$)+1+MND(red($B$))=MNA($f(A)$)+1+MNA($f($red($B$)))

=MNA($f(\pi)$).

Similarly, suppose that rar($A$) is even. Then, because rar($\pi$)=ldr($f(\pi)$), the element $n$ does not contribute to MNA($\pi$) and MND($f(\pi)$), so we have

MNA($\pi$)=MNA($A$)+MNA(red($B$))=MND($f(A)$)+MND($f($red($B$)))

=MND($f(\pi)$).

Finally, if rar($A$) is odd, then $n$ contributes one extra non-overlapping ascent in $\pi$ and one extra non-overlapping descent in $f(\pi)$:

MNA($\pi$)=MNA($A$)+1+MNA(red($B$))=MND($f(A)$)+1+MND($f($red($B$)))

=MND($f(\pi)$).

\noindent \fbox{$A = \varepsilon$, $ B\neq \varepsilon$.}

ldr($\pi$)=1+ldr(red($B$))=1+rar($f$(red($B$)))=rar($f(\pi)$). 

rar($\pi$)=rar(red$(B)$)=ldr($f($red$(B)$))=ldr($f(\pi)$). 

lmin($\pi$)=1+lmin(red($B$))=1+rmin($f$(red($B$)))=rmin($f(\pi)$). 

rmin($\pi$)=rmin(red$(B)$)=lmin($f($red$(B)$))=lmin($f(\pi)$). 

asc($\pi$)=asc(red($B$))=des($f($red($B$)))=des($f(\pi)$).

des($\pi$)=1+des(red($B$))=1+asc($f($red($B$)))=asc($f(\pi)$).

Now, suppose that ldr(red$(B)$) is even. Then, because ldr($\pi$)=rar($f(\pi)$), the element $n$ does not contribute to MND($\pi$) and MNA($f(\pi)$), so we have

MND($\pi$)=MND(red($B$))=MNA($f($red($B$)))=MNA($f(\pi)$).

However, if ldr(red$(B)$) is odd, then $n$ contributes one extra non-overlapping descent in $\pi$ and one extra non-overlapping ascent in $f(\pi)$:

MND($\pi$)=1+MND(red($B$))=1+MNA($f($red($B$)))=MNA($f(\pi)$).

Finally, MNA($\pi$)=MNA(red($B$))=MND($f($red($B$)))=MND($f(\pi)$).

\noindent \fbox{$A\neq \varepsilon$, $ B= \varepsilon$.}

ldr($\pi$)=ldr($A$)=rar($f(A$))=rar($f(\pi)$). 

rar($\pi$)=1+rar($A$)=1+ldr($f(A)$)=ldr($f(\pi)$). 

lmin($\pi$)=lmin($A$)=rmin($f(A$))=rmin($f(\pi)$). 

rmin($\pi$)=1+rmin($A$)=1+lmin($f(A)$)=lmin($f(\pi)$). 

asc($\pi$)=asc($A$)+1=des($f(A)$)+1=des($f(\pi)$).

des($\pi$)=des($A$)=asc($f(A)$)=asc($f(\pi)$).

Now, MND($\pi$)=MND($A$)=MNA($f(A)$)=MNA($f(\pi)$).

Further, suppose that rar($A$) is even. Then, because rar($\pi$)=ldr($f(\pi)$), the element $n$ does not contribute to MNA($\pi$) and MND($f(\pi)$), so we have

MNA($\pi$)=MNA($A$)=MND($f(A)$)=MND($f(\pi)$).

Finally, if rar($A$) is odd, then $n$ contributes one extra non-overlapping ascent in $\pi$ and one extra non-overlapping descent in $f(\pi)$:

MNA($\pi$)=MNA($A$)+1=MND($f(A)$)+1=MND($f(\pi)$).

This completes our proof.
\end{proof}

\section{MND on S-S permutations and binary trees}\label{sec-binary-trees}

To achieve the desired result, we need the following theorem, the proof of which introduces the bijection $g$ on S-S permutations.

\begin{thm}\label{ldr-rmax-thm} The statistics ldr and rmax are equidistributed on S-S permutations. \end{thm}

\begin{proof} It is easy to see using the fact that S-S permutations are precisely 231-avoiding permutations that the structure of any S-S permutation  is $x_{\ell}x_{\ell-1}\cdots x_2x_1A_1A_2\cdots A_{\ell-1}A_{\ell}$ where 
\begin{itemize}
\item $x_i$'s are left-to-right minima (in particular, $x_1=1$);
\item $1=x_1<A_1<x_2<A_2<\cdots <x_{\ell}<A_{\ell}$ (i.e., the elements in $A_i$ are a permutation of the set $\{x_i+1,x_i+2,\cdots, x_{i+1}-1\}$ assuming $x_{\ell+1}:=n+1$);
\item each $A_i$ is a (possibly empty) $231$-avoiding permutation. 
\end{itemize} 

On the other hand, to avoid the pattern $231$, the structure of any S-S permutation is $B_1nB_2(n-1)\cdots B_{\ell}(n-\ell+1)$ where 
\begin{itemize}
\item $n(n-1)\cdots (n-\ell+1)$ is the sequence of right-to-left maxima formed by the largest elements;
\item $B_1<B_2<\cdots<B_{\ell}$; 
\item each $B_i$ is a (possibly empty) $231$-avoiding permutation. 
\end{itemize}
The map $g$ takes an S-S permutation $\pi$ with rmax$(\pi)=\ell$ and the structure described above and sends it to the S-S permutation $g(\pi)$ with ldr$(g(\pi))=\ell$ and the structure described above so that red$(A_i)=$red$(B_i)$ for $i=1,2,\ldots,\ell$. That is straightforward to see that $g$ is injective and surjective and hence bijective proving the equidistribution of ldr and rmax.  
\end{proof} 

The fact that the distribution of MND on S-S permutations is given by (\ref{231-MND}) follows from the next theorem (recall the definition of the statistic X in Section~\ref{trees-sub}). We refer to Figure~\ref{example-thm-binary-fig} for an illustration of the steps in the proof of Theorem~\ref{thm-biary-trees-X-MND}.

\begin{figure}
\begin{center}
\begin{tabular}{cccccccc}
$T=$ & 
\hspace{-0.8cm} \begin{minipage}[c]{8em}\scalebox{0.8}{
\begin{tikzpicture}[node distance=1cm,auto,main node/.style={circle,draw,inner sep=1pt,minimum size=2pt}]

\node[style={circle,fill,draw,inner sep=2pt,minimum size=2pt}] (6) {};
\node[style={circle,fill,draw,inner sep=2pt,minimum size=2pt}] (5) [below left of=6] {};
\node[style={circle,fill,draw,inner sep=2pt,minimum size=2pt}] (1) [below right of=6] {};
\node[style={circle,fill,draw,inner sep=2pt,minimum size=2pt}] (4) [xshift=-0.1cm, below left of=5] {};
\node[style={circle,fill,draw,inner sep=2pt,minimum size=2pt}] (3) [xshift=-0.2cm, below right of=5] {};
\node[style={circle,fill,draw,inner sep=2pt,minimum size=2pt}] (2) [xshift=0.3cm, below left of=3] {};
\node[style={circle,fill,draw,inner sep=2pt,minimum size=2pt}] (a) [xshift=-0.2cm, below right of=3] { };
\node[style={circle,fill,draw,inner sep=2pt,minimum size=2pt}] (b) [xshift=0.3cm, below left of=1] { };
\node[style={circle,fill,draw,inner sep=2pt,minimum size=2pt}] (c) [xshift=-0.2cm, below right of=1] { };
\node[style={circle,fill,draw,inner sep=2pt,minimum size=2pt}] (d) [xshift=0.2cm, below left of=4] { };
\node[style={circle,fill,draw,inner sep=2pt,minimum size=2pt}] (e) [xshift=-0.3cm, below right of=4] { };
\node[style={circle,fill,draw,inner sep=2pt,minimum size=2pt}] (f) [xshift=0.2cm, below left of=2] { };
\node[style={circle,fill,draw,inner sep=2pt,minimum size=2pt}] (g) [xshift=-0.2cm, below right of=2] { };

\path
(2) edge (f)
(2) edge (g)
(4) edge (d)
(4) edge (e)
(6) edge (5)
(6) edge (1)
(5) edge (4)
(5) edge (3)
(3) edge (2)
(3) edge (a)
(1) edge (b)
(1) edge (c);
\end{tikzpicture}
}
\end{minipage}
& \hspace{-0.5cm} $\mapsto$ & 
\hspace{-0.7cm}
\begin{minipage}[c]{8em}\scalebox{0.8}{
\begin{tikzpicture}[node distance=1cm,auto,main node/.style={circle,draw,inner sep=1pt,minimum size=2pt}]

\node[main node] (6) {6};
\node[main node] (5) [below left of=6] {5};
\node[main node] (1) [below right of=6] {1};
\node[main node] (4) [xshift=-0.1cm, below left of=5] {4};
\node[main node] (3) [xshift=-0.2cm, below right of=5] {3};
\node[main node] (2) [xshift=0.3cm, below left of=3] {2};
\node[style={circle,fill,draw,inner sep=2pt,minimum size=2pt}] (a) [xshift=-0.2cm, below right of=3] { };
\node[style={circle,fill,draw,inner sep=2pt,minimum size=2pt}] (b) [xshift=0.3cm, below left of=1] { };
\node[style={circle,fill,draw,inner sep=2pt,minimum size=2pt}] (c) [xshift=-0.2cm, below right of=1] { };
\node[style={circle,fill,draw,inner sep=2pt,minimum size=2pt}] (d) [xshift=0.2cm, below left of=4] { };
\node[style={circle,fill,draw,inner sep=2pt,minimum size=2pt}] (e) [xshift=-0.3cm, below right of=4] { };
\node[style={circle,fill,draw,inner sep=2pt,minimum size=2pt}] (f) [xshift=0.2cm, below left of=2] { };
\node[style={circle,fill,draw,inner sep=2pt,minimum size=2pt}] (g) [xshift=-0.2cm, below right of=2] { };

\path
(2) edge (f)
(2) edge (g)
(4) edge (d)
(4) edge (e)
(6) edge (5)
(6) edge (1)
(5) edge (4)
(5) edge (3)
(3) edge (2)
(3) edge (a)
(1) edge (b)
(1) edge (c);
\end{tikzpicture}
}
\end{minipage}
& \hspace{-0.5cm} \begin{tabular}{c} {\small $h$} \\[-2mm] $\mapsto$ \end{tabular} & 
\hspace{-0.4cm}163254 & 
\hspace{-0.4cm} \begin{tabular}{c} {\small $h^+$} \\[-2mm] $\mapsto$ \end{tabular} & \hspace{-0.5cm} 165243
\end{tabular}
\caption{An illustration of the steps in the proof of Theorem~\ref{thm-biary-trees-X-MND}. Note that $X(T)=\lfloor\frac{3}{2}\rfloor+\lfloor\frac{2}{2}\rfloor+\lfloor\frac{1}{2}\rfloor=2=$MND($165243$).}\label{example-thm-binary-fig}
\end{center}
\end{figure}
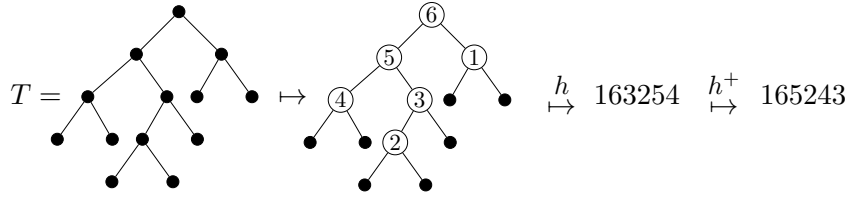

\begin{thm}\label{thm-biary-trees-X-MND} The set $\{ T\ |\ T \mbox{ is a full binary tree with }2n \mbox{ edges; }X(T)=k\}$ is in one-to-one correspondence with  the set \\ $\{\pi\ |\ \pi \mbox{ is an S-S permutation of length }n; \mbox{MND}(\pi)=k\}$.
\end{thm}

\begin{proof}  Label the non-leaf vertices in a full binary tree $T$ with $2n$ edges using the pre-oder traversal (that is, visiting the leftmost yet unvisited vertices first) and decreasing order of labels starting from $n$. Then $T=T_LnT_R$ where $n$ is the root, $T_L$ is the left subtree and $T_R$ is the right subtree. Now, define recursively the map $h$ from labeled full binary trees to 231-avoiding permutations by $h(T):=h(T_R)nh(T_L)$, and as the base case, $h$ maps the one-vertex tree to $\varepsilon$. In particular, if $h(T_L)\neq\varepsilon$,  the permutation $h(T_L)$ is formed by the largest elements below $n$ and hence the outcome is a 231-avoiding permutation. Note that if $h(T)=AnB$, then $h^+(T):=Ang(B)$ where $g$ is defined in the proof of Theorem~\ref{ldr-rmax-thm} and by $g(B)$ we actually mean  first computing $B'=g(\mbox{red}(B))$ and then taking the order-isomorphic to $B'$ permutation formed by the largest elements below $n$. Clearly, $h^+$ is a bijection. Also, note that lpath($T$)=rmax($h(T)$), where lpath($T$) is the length of the leftmost path in $T$, that is, the number of edges on the path from the root to the leftmost leaf.

Next, we prove by induction on $n$, with the obvious base cases of $n=0,1$, that if $X(T)=k$ then MND$(h^+(T))=k$. We consider two cases given by the parity of lpath($T$). Suppose that lpath($T$) is odd. Then, by definition of $X$, $X(T)=X(T_LnT_R)=$
$$X(T_L)+X(T_R)=\mbox{MND}(h^+(T_L))+\mbox{MND}(h^+(T_R))=\mbox{MND}(h^+(T))$$ where the last equality follows from the fact that rmax($h(T_L)$)$=$ldr($h^+(T_L)$) is even and the element $n$ in $h^+(T)$ does not increase MND. On the other hand, if lpath($T$) is even then $X(T)=X(T_LnT_R)=$
$$X(T_L)+1+X(T_R)=\mbox{MND}(h^+(T_L))+1+\mbox{MND}(h^+(T_R))=\mbox{MND}(h^+(T))$$ where the last equality follows from the fact that rmax($h(T_L)$)$=$ldr($h^+(T_L)$) is odd and the element $n$ in $h^+(T)$, together with the leftmost decreasing run in $h^+(T_L)$, creates an extra occurrence of a non-overlapping descent in $h^+(T)$, and the theorem is proved.
\end{proof}

\section{MND on S-S permutations and marked trees}\label{sec-trees}

This section gives an alternative proof (to the one presented in Section~\ref{sec-binary-trees}) of the fact that the distribution of MND on S-S permutations is given by (\ref{231-MND}), and we present it in the next theorem. (Recall relevant definitions in Section~\ref{trees-sub}). The bijection in the proof of Theorem~\ref{thm-MND} is the most involved out of the three bijections in this paper. 

\begin{thm}\label{thm-MND} The set $$\{ T\ |\ T \mbox{ is a rooted plane tree with }n \mbox{ edges and }k\mbox{ marked vertices}\}$$ is in one-to-one correspondence with  the set  $$\{\pi\ |\ \pi \mbox{ is an S-S permutation of length }n; \mbox{MND}(\pi)=k\}.$$\end{thm}

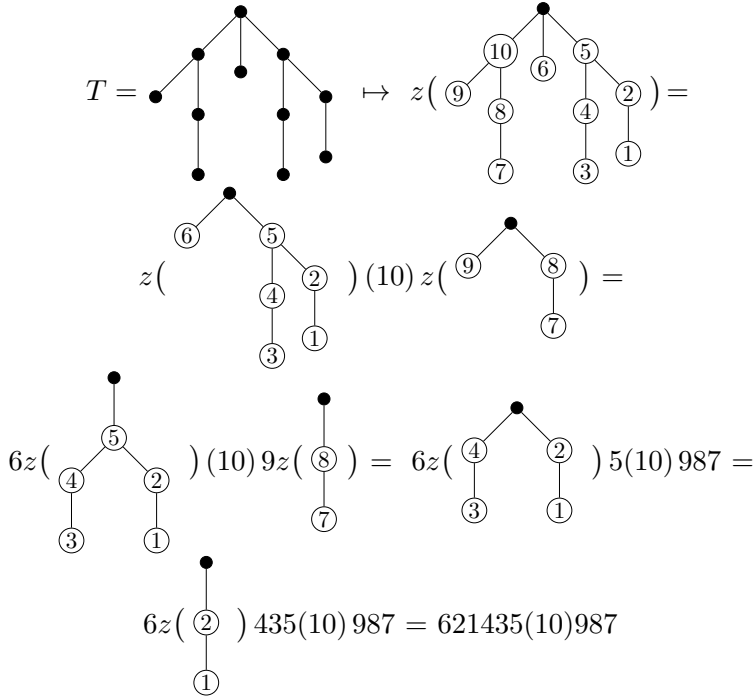
\begin{figure}[h]
\begin{center}
\begin{tabular}{ccccc}
$T=$ & 
\hspace{-0.5cm} 
\begin{minipage}[c]{8em}\scalebox{0.8}{
\begin{tikzpicture}[node distance=1cm,auto,main node/.style={circle,draw,inner sep=1pt,minimum size=2pt}]

\node[style={circle,fill,draw,inner sep=2pt,minimum size=2pt}] (11) {};
\node[style={circle,fill,draw,inner sep=2pt,minimum size=2pt}] (10) [below left of=11] {};
\node[style={circle,fill,draw,inner sep=2pt,minimum size=2pt}] (6) [below of=11] {};
\node[style={circle,fill,draw,inner sep=2pt,minimum size=2pt}] (5) [below right of=11] {};
\node[style={circle,fill,draw,inner sep=2pt,minimum size=2pt}] (9) [below left of=10] {};
\node[style={circle,fill,draw,inner sep=2pt,minimum size=2pt}] (8) [below of=10] {};
\node[style={circle,fill,draw,inner sep=2pt,minimum size=2pt}] (7) [below of=8] {};
\node[style={circle,fill,draw,inner sep=2pt,minimum size=2pt}] (4) [below of=5] {};
\node[style={circle,fill,draw,inner sep=2pt,minimum size=2pt}] (3) [below of=4] {};
\node[style={circle,fill,draw,inner sep=2pt,minimum size=2pt}] (2) [below right of=5] {};
\node[style={circle,fill,draw,inner sep=2pt,minimum size=2pt}] (1) [below of=2] {};

\path
(5) edge (4)
(5) edge (2)
(4) edge (3)
(2) edge (1)
(8) edge (7)
(10) edge (9)
(10) edge (8)
(11) edge (10)
(11) edge (6)
(11) edge (5);
\end{tikzpicture}
}
\end{minipage}
& \hspace{-0.7cm} $\mapsto$  & 
\hspace{-0.3cm}
$z$\big(\begin{minipage}[c]{7.5em}\scalebox{0.8}{
\begin{tikzpicture}[node distance=1cm,auto,main node/.style={circle,draw,inner sep=1pt,minimum size=2pt}]

\node[style={circle,fill,draw,inner sep=2pt,minimum size=2pt}] (11) {};
\node[main node] (10) [below left of=11] {10};
\node[main node] (6) [below of=11] {6};
\node[main node] (5) [below right of=11] {5};
\node[main node] (9) [below left of=10] {9};
\node[main node] (8) [below of=10] {8};
\node[main node] (7) [below of=8] {7};
\node[main node] (4) [below of=5] {4};
\node[main node] (3) [below of=4] {3};
\node[main node] (2) [below right of=5] {2};
\node[main node] (1) [below of=2] {1};

\path
(5) edge (4)
(5) edge (2)
(4) edge (3)
(2) edge (1)
(8) edge (7)
(10) edge (9)
(10) edge (8)
(11) edge (10)
(11) edge (6)
(11) edge (5);
\end{tikzpicture}
}
\end{minipage}\hspace{-0.7mm}\big)
& \hspace{-0.5cm} = 
\end{tabular}

\begin{tabular}{ccc}
\hspace{-0.3cm}
$z$\big(\begin{minipage}[c]{8em}\scalebox{0.8}{
\begin{tikzpicture}[node distance=1cm,auto,main node/.style={circle,draw,inner sep=1pt,minimum size=2pt}]

\node[style={circle,fill,draw,inner sep=2pt,minimum size=2pt}] (11) {};
\node[main node] (6) [below left of=11] {6};
\node[main node] (5) [below right of=11] {5};
\node[main node] (4) [below of=5] {4};
\node[main node] (3) [below of=4] {3};
\node[main node] (2) [below right of=5] {2};
\node[main node] (1) [below of=2] {1};

\path
(5) edge (4)
(5) edge (2)
(4) edge (3)
(2) edge (1)
(11) edge (6)
(11) edge (5);
\end{tikzpicture}
}
\end{minipage}
\hspace{-0.8cm}\big)
& \hspace{-0.5cm} (10) & \hspace{-0.5cm}
$z$\big(\begin{minipage}[c]{6em}\scalebox{0.8}{
\begin{tikzpicture}[node distance=1cm,auto,main node/.style={circle,draw,inner sep=1pt,minimum size=2pt}]

\node[style={circle,fill,draw,inner sep=2pt,minimum size=2pt}] (10) {};
\node[main node] (9) [below left of=10] {9};
\node[main node] (8) [below right of=10] {8};
\node[main node] (7) [below of=8] {7};

\path
(8) edge (7)
(10) edge (9)
(10) edge (8);
\end{tikzpicture}
}
\end{minipage}
\hspace{-0.7cm}\big) =

\end{tabular}

\begin{tabular}{cccccc}
\hspace{-0.3cm}
6$z$\big(\begin{minipage}[c]{6.5em}\scalebox{0.8}{
\begin{tikzpicture}[node distance=1cm,auto,main node/.style={circle,draw,inner sep=1pt,minimum size=2pt}]

\node[style={circle,fill,draw,inner sep=2pt,minimum size=2pt}] (11) {};
\node[main node] (5) [below of=11] {5};
\node[main node] (4) [below left of=5] {4};
\node[main node] (3) [below of=4] {3};
\node[main node] (2) [below right of=5] {2};
\node[main node] (1) [below of=2] {1};

\path
(5) edge (4)
(5) edge (2)
(4) edge (3)
(2) edge (1)
(11) edge (5);
\end{tikzpicture}
}
\end{minipage}
\hspace{-0.8cm}\big)
& \hspace{-0.5cm} (10) & \hspace{-0.5cm}
9$z$\big(\begin{minipage}[c]{3em}\scalebox{0.8}{
\begin{tikzpicture}[node distance=1cm,auto,main node/.style={circle,draw,inner sep=1pt,minimum size=2pt}]

\node[style={circle,fill,draw,inner sep=2pt,minimum size=2pt}] (10) {};
\node[main node] (8) [below of=10] {8};
\node[main node] (7) [below of=8] {7};

\path
(8) edge (7)
(10) edge (8);
\end{tikzpicture}
}
\end{minipage}
\hspace{-0.7cm}\big) =
&
\hspace{-0.3cm}
6$z$\big(\begin{minipage}[c]{6.5em}\scalebox{0.8}{
\begin{tikzpicture}[node distance=1cm,auto,main node/.style={circle,draw,inner sep=1pt,minimum size=2pt}]

\node[style={circle,fill,draw,inner sep=2pt,minimum size=2pt}] (5) {};
\node[main node] (4) [below left of=5] {4};
\node[main node] (3) [below of=4] {3};
\node[main node] (2) [below right of=5] {2};
\node[main node] (1) [below of=2] {1};

\path
(5) edge (4)
(5) edge (2)
(4) edge (3)
(2) edge (1);
\end{tikzpicture}
}
\end{minipage}
\hspace{-0.8cm}\big)
& \hspace{-0.5cm} 5(10) & \hspace{-0.5cm}
987 =
\end{tabular}

\begin{tabular}{ccc}
\hspace{-0.3cm}
6$z$\big(\begin{minipage}[c]{3.5em}\scalebox{0.8}{
\begin{tikzpicture}[node distance=1cm,auto,main node/.style={circle,draw,inner sep=1pt,minimum size=2pt}]

\node[style={circle,fill,draw,inner sep=2pt,minimum size=2pt}] (5) {};
\node[main node] (2) [below of=5] {2};
\node[main node] (1) [below of=2] {1};

\path
(5) edge (2)
(2) edge (1);
\end{tikzpicture}
}
\end{minipage}
\hspace{-0.8cm}\big)
& \hspace{-0.5cm} 435(10) & \hspace{-0.5cm}
987 = 621435(10)987
\end{tabular}

\caption{Computing $z(T)$ in the proof of Theorem~\ref{thm-MND} that results in the permutation $\pi=621435(10)987$. Note that ldr($\pi$) is odd, which is consistent with $T$'s root having a leaf as a child, and M($T$)=MND($\pi$)=4 as the marked vertices in $T$ are 2, 4, 8, 10.}\label{example-thm-trees-fig}
\end{center}
\end{figure}

\begin{proof} Let M$(T)$ denote the number of marked vertices in a rooted plane tree $T$. 

We define the map  $z$ from the set of trees to the set of permutations as follows. The one vertex tree (with no edges, $n=0$) is mapped by $z$ to the empty permutation $\varepsilon$. For the rest of the recursive description of $z$ it is convenient to label all non-root vertices in each rooted plane tree using the pre-order traversal (beginning from the root and going in the leftmost available direction) and starting assigning the largest label $n$ (to the leftmost child of the root) and then assigning the labels in the descending order (similarly to our labelling in the proof of Theorem~\ref{thm-biary-trees-X-MND}).  To define the image $z(T)$ for a tree $T$ in the set with $n\geq 1$, we distinguish the following five cases in which we check by induction on $n$, with the trivial base case of $n=0$, that 
\begin{itemize}
\item ldr$(z(T))$ is even if the root of $T$ has no leaf as a child and ldr$(z(T))$ is odd otherwise, and
\item M$(T)=$MND($z(T)$).
\end{itemize}
We refer to Figure~\ref{example-thm-trees-fig} for an example involving all cases. So, the cases are:
\begin{itemize}
\item[(1)] $T$ has a single child $r_1$ labeled by $n$ that is the root of the subtree $A$ and no children of $r_1$ is a leaf. Let $z(T):=z(A)n$. Note that ldr$(z(A))$ is even by the inductive hypothesis, and ldr$(z(T))$ is even too. Also, MND$(z(T))$=MND$(z(A))$=M$(A)$=M$(T)$ since $r_1$ is not marked in~$T$.
\item[(2)] $T$ has a single child $r_1$ labeled by $n$ that is the root of the subtree $A$ and at least one child of $r_1$ is a leaf, so $r_1$ is marked. Let $z(T):=nz(A)$.  Note that ldr$(z(A))$ is odd by the inductive hypothesis, and hence ldr$(z(T))$ is even. Then the element $n$ in $z(T)$ contributes an extra non-overlapping descent, which is consistent with $r_1$ being marked:  MND$(z(T))$=1+MND$(z(A))$=1+M$(A)$=M$(T)$.
\item[(3)] The leftmost child $r_1$ (labelled by $n$) of the root is a leaf, the root has at least one more child, but none of the  other children of the root is a leaf. Let the subtree $A$ be $T$ without $r_1$. Then $z(T):=nz(A)$. Note that ldr$(z(A))$ is even by the inductive hypothesis, and hence ldr$(z(T))$ is odd (which makes the outcome here be different from that in case (2)). Clearly, MND$(z(T))$=MND$(z(A))$=M$(A)$=M$(T)$.
\item[(4)]  The root has no leaves, the leftmost child $r_1$ (labelled by $n$) of the root is the root of the subtree $B$ (with at least one edge), and the rest of $T$ is the subtree $A$ with at least one edge (and root having no leaves as children). Let $z(T):=z(A)nZ(B)$ where $Z(B)$ is formed by the largest elements below $n$. Note that  ldr$(z(T))$=ldr$(z(A))$ is even.  Furthermore,
\begin{itemize}
\item if $r_1$ has no leaf among its children, then ldr($z(B)$) is even and MND$(z(T))$=MND$(z(A))$+MND$(z(B))$=M$(A)$+M$(B)$=M$(T)$;
\item if $r_1$ has a leaf as a child, then ldr($z(B)$) is odd, $r_1$ is marked and MND$(z(T))$=MND$(z(A))$+1+MND$(z(B))$=M$(A)$+1+M$(B)$=M$(T)$.
\end{itemize}
We remark that it is essential for our goals to consider cases (3) and (4) separately rather than allowing $B$ be the one-vertex tree in case~(4). 
\item[(5)] The root has a leaf, the leftmost child $r_1$ (labelled by $n$) of the root is the root of the possibly one-vertex subtree $B$, and the rest of $T$ is the subtree $A$ with at least one edge and root having a leaf as a child. Let $z(T):=z(A)nZ(B)$ where $Z(B)$  is formed by the largest elements below $n$. Note that  ldr$(z(T))$=ldr$(z(A))$ is odd.  Furthermore,
\begin{itemize}
\item if $r_1$ has no leaf among its children, then ldr($z(B)$) is even (possibly 0) and MND$(z(T))$=MND$(z(A))$+MND$(z(B))$=M$(A)$+M$(B)$ =M$(T)$;
\item if $r_1$ has a leaf as a child, then ldr($z(B)$) is odd, $r_1$ is marked and MND$(z(T))$=MND$(z(A))$+1+MND$(z(B))$=M$(A)$+1+M$(B)$=M$(T)$.
\end{itemize}
Note that if $B$ is the one-vertex tree in case (5), this case is different from case (3). 
\end{itemize}

It is straightforward to see that the cases above are disjoint and any given rooted plane tree belongs to one of the cases. In particular, if the root of a tree has a single child then we are in case (1) or case (2). Moreover, the map $z$ is clearly injective and surjective (as every 231-avoiding permutation appears exactly once as the image which is easy to see by considering where $n$ is placed) and hence $z$ is a bijection giving the desired result. 
\end{proof}

\section{Directions for further research}\label{open-directions}

In this paper, we find the distribution of the maximum number of non-overlapping descents and ascents over S-S permutations, which are precisely 231-avoiding permutations. Using trivial bijections on permutations (reverse, complement, and their composition), our results provide the distribution of MND and MNA (given by (\ref{231-MND})) over $p$-avoiding permutations, where $p\in\{132, 213, 231, 312\}$.  Additionally note that MND on 321-avoiding permutations (equivalently, MNA on 123-avoiding permutations) is given by (\ref{321-des}), because 321-avoiding permutations do not have occurrences of overlapping descents. However, finding distribution of MNA on 321-avoiding permutations is still an open problem (see Table~\ref{tab-MNA-321} for the respective distribution).

\begin{table}
\begin{center}
\begin{tabular}{|c|c|c|c|c|c|c|c|}
\hline
\diagbox{$n$}{$k$}&  0  & 1  & 2 &  3 & 4 & 5 & 6 \\
\hline
1&  1 & & & & & & \\
\hline
2&  1 & 1 & & & & & \\
\hline
3& 0 & 5 & & & & & \\
\hline
4 & 0  & 8 & 6 & & & & \\
\hline
5&  0 & 5  & 37 & & & & \\
\hline
6& 0  & 0 & 89 & 43 & & & \\
\hline
7&  0 & 0 &98  & 331& & & \\
\hline
8 & 0  & 0 & 42 & 1036& 352 & & \\
\hline
9 &  0 & 0 & 0 & 1644 & 3218 & & \\
\hline
10 &  0 & 0 &0  & 1320 & 12362 & 3114 & \\
\hline
11 &  0 & 0 & 0 & 429 & 25498 & 32859 & \\
\hline
12  &  0 & 0 & 0 & 0 & 29744 & 149264 & 29004 \\
\hline
\end{tabular}
  \caption{Distribution of MNA (resp., MND) on 321-avoiding (resp., 123-avoiding) permutations, where $n$ is the length of permutations and  $k$ is the number of occurrences of the statistic.}
\label{tab-MNA-321}
\end{center}
\end{table}

Open directions for research include finding distributions of MND and MNA over other (pattern-avoiding) classes of permutations, in particular, over 2-stack sortable permutations (permutations sortable by two applications of the  stack-sorting operator $\mathcal{S}$). 

\vskip 3mm
\noindent {\bf Acknowledgments.}
The first author is grateful to Tianjin Normal University for its hospitality. The work of the second author was supported by the National Science Foundation of China (No. 12171362).


\begin{thebibliography}{10}

\bibitem{BBS2010} M. Barnabei, F. Bonetti, M. Silimbani. The descent statistic on 123-avoiding permutations. {\em S\'em. Loth. de Comb.} {\bf 63} (2010), Article B63a.

\bibitem{Bona} M. B\'{o}na. A survey of stack sortable permutations. {\em 50 years of combinatorics, graph theory, and computing}, Discrete Math. Appl. (Boca Raton), CRC Press, Boca Raton, FL, (2020) 55--72.

\bibitem{BKLPRW} M. Bukata, R. Kulwicki, N. Lewandowski, L. Pudwell, J. Roth, T. Wheeland. Distributions of statistics over pattern-avoiding permutations, {\em J. Integer Sequences} {\bf 22} (2019), Article 19.2.6.

\bibitem{Callan2005} D. Callan. Some identities for the Catalan and Fine numbers, {\em S\'{e}m. Lothar. Combin.} {\bf 53} (2004) Art. B53e, 16pp.

    \bibitem{Kitaev2011Patterns}
    S. Kitaev. Patterns in permutations and words, Monographs in Theoretical Computer Science. An EATCS Series. Springer,  2011.

 \bibitem{Kitaev2005} S. Kitaev. Partially ordered generalized patterns, {\em Discrete Math.} {\bf 298} (2005), 212--229.

    \bibitem{Knuth1969Art}
    D.~E. Knuth. The {A}rt of {C}omputer {P}rogramming. {V}ol. 1: {F}undamental
    {A}lgorithms, Second printing, Addison-Wesley Publishing Co., Reading,
    Mass.-London-Don Mills, Ont, 1969.

\bibitem{oeis} N.~J.~A.~Sloane. The Online Encyclopedia of Integer Sequences. Published electronically at \url{http://oeis.org}.
    
     \bibitem{Stanley1989} R. P. Stanley. Enumerative combinatorics. Vol. 2, volume 62 of Cambridge Studies in Advanced Mathematics. Cambridge University Press, Cambridge, 1999.
     
     \bibitem{Sun2010} Y. Sun. A simple bijection between binary trees and colored ternary trees, {\em Electron. J. Combin.} {\bf 17} (2010) 1, N20, 5pp.

    \bibitem{West1990Permutations}
    J.~West. Permutations with forbidden subsequences and stack-sortable
    permutations, ProQuest LLC, Ann Arbor, MI, 1990.
    \newblock Thesis (Ph.D.)--Massachusetts Institute of Technology.

\end{thebibliography}
\end{document}